\documentclass[12pt]{article}

\def\update{Aug 22, 2008}

\usepackage{
amssymb,
amsthm}
\usepackage{graphicx}

\usepackage{amsmath}
\usepackage{amsfonts}

 \usepackage[colorlinks,linkcolor=black]{hyperref}

\newtheorem{definition}{\textbf{Definition}}

\newtheorem{theorem}{\textbf{Theorem}}
\newtheorem{corollary}{\textbf{Corollary}}

\newtheorem{lemma}{\textbf{Lemma}}
\newtheorem{exe}{\textbf{Example}}

\def\N {\mathbb{N}}
\def\Z {\mathbb{Z}}
\def\QQ {\mathbb{Q}}

\def\C {\mathbb{C}}

\begin{document}

\noindent
AWS 2008 Project
\hfill
{updated: \it \update}

\begin{center}

\medskip

{
\LARGE\mathversion{bold}
Algebraic Values of Transcendental Functions at Algebraic Points
}
\medskip

 \large

\medskip

{\it by }
\medskip

  Jingjing Huang, Diego Marques and Martin Mereb

\end{center}

\bigskip

\abstract{In this paper, the authors will prove that any subset of
$\overline{\QQ}$ can be the exceptional set of some transcendental
entire function. Furthermore, we could generalize this theorem to a
much more general version and present a unified proof.}


\section{Introduction and the main result}

\hspace{0.5 cm} In 1886, Weierstrass gave an example of a
transcendental entire function which takes rational values at all
rational points. He also suggested that there exist transcendental
entire functions which take algebraic values at any algebraic point.
Later, in \cite{Sta1}, St$\ddot{\mbox{a}}$ckel proved that for each
countable subset $\Sigma \subseteq \C$ and each dense subset $T
\subseteq \C$, there is a transcendental entire function $f$ such
that $f(\Sigma) \subseteq T$. Another construction due to
St$\ddot{\mbox{a}}$ckel produces an entire function $f$ whose
derivatives $f^{(s)}$, for $s = 0,1, 2,\dots,$ all map
$\overline{\QQ}$ into $\overline{\QQ}$, see \cite{Sta2}. A more
thorough discussion on this subject can be found in \cite{Mah}.
There are recent results due to Andrea Surroca on the number of
algebraic points where a transcendental analytic function takes
algebraic values, see \cite{Su2}. We were able to generalize these
two results of St$\ddot{\mbox{a}}$ckel to the following general
theorem.
\begin{theorem}
Given a countable subset $A \subseteq \C$  and for each integer $s
\geq 0$ with $\alpha \in A$, fix a dense subset $E_{\alpha,s}
\subseteq \C$. Then there exists a transcendental entire function
$f:\C \to \C$ such that $f^{(s)}(\alpha) \in E_{\alpha, s}$, for all
$\alpha\in A$ and all $s\geq 0$.
\end{theorem}
Let $f$ be given, and denote by $S_f$ the set of all algebraic
points $\alpha \in \C$, for which $f(\alpha)$ is also algebraic. An
interesting problem is to determine properties of $S_f$, which we
name as the exceptional set of $f$. In the conclusion we will show
that for any $A \subseteq \overline{\QQ}$ there is a transcendental
entire function $f$ such that $A$ is the exceptional set of $f$.

Without referring to Theorem 1, we have the following special
examples:
\begin{exe}
 Arbitrary finite subsets of algebraic
numbers are easily seen to be exceptional. For instance, if $f_1(z)
= e^{(z-\alpha_1) \cdots (z-\alpha_k)},$ then the Hermite-Lindemann
theorem implies $S_{f_1} = \{\alpha_1, \ldots, \alpha_k\}$. If
$f_2(z) = e^z + e^{z+1} \mbox{and}\ f_3(z) = e^{z\pi + 1}$, then the
Lindemann-Weierstrass and Baker theorems imply $S_{f_2} = S_{f_3} =
\varnothing$.
\end{exe}
\begin{exe}
Some well-known infinite sets are also exceptional, for instance, if
$f_4(z) = 2^z, f_5(z) = e^{i\pi z}$, then $S_{f_4} = S_{f_5} = \QQ$,
by the Gelfond-Schneider theorem.
\end{exe}
\begin{exe}
Assuming Schanuel's conjecture to be true, it is easy to prove that
if $f_6(z) = \sin(\pi z)e^z, f_7(z) = 2^{3^z}$ and $f_8(z) =
2^{2^{2^{z - 1}}}$, then $S_{f_6} = S_{f_7} = \Z$ and $S_{f_8} =
\N.$
\end{exe}

These examples are just special case of our Theorem 1, hitherto can
be proved uniformly here.


\section{Preliminary results}
Before going to the proof of the theorem, we need couple of lemmas.

\begin{lemma}
Let $\{P_n(z)\}_{n\geq 0}$ be a sequence of complex polynomials,
where $\deg P_n=n$. And let $\{C_n\}_{n\geq 0}$ be a sequence of
positive constants providing that $|P_n(z)|\leq C_n\max\{|z|,1\}^n$.
If a sequence of complex numbers $\{a_n\}_{n \geq 0}$ satisfies
$|a_n|\leq \frac{1}{C_n n!}$, then the series $\sum_{n = 0}^\infty
a_nP_n(z)$ converges absolutely and uniformly on any compact sets,
particularly this gives an entire function.
\end{lemma}

\begin{proof}

When $|a_n|\leq \frac{1}{C_n n!}$, we have:

$$
\sum_{n = 0}^\infty |a_n||P_n(z)| \leq \sum_{n = 0}^\infty
\frac{1}{C_n n!}C_n\max\{|z|,1\}^n \leq \exp(\max\{|z|,1\}),
$$
so $\sum_{n = 0}^\infty a_nP_n(z)$ converges absolutely and
uniformly on any compact sets. Therefore this series will produce an
entire function.
\end{proof}

Now, let's enumerate the set $A$ in Theorem 1 as
$\{\alpha_1,\alpha_2,\alpha_3,\cdots\}$

    For $n\geq 1$, setting $n = 1+2+3+\cdots+m_n + j_n$, where $m_n \geq 0$ and $1 \leq j_n \leq m_n + 1$.  Next, construct a sequence of polynomials as follows
\begin{center}
$P_0(z) = 1$ and define recursively $P_n(z) =
(z-\alpha_{j_n})P_{n-1}(z)$ for $n\geq 1$
\end{center}
Here we list the first few polynomials:
  \begin{eqnarray*}
P_0(z) & = & 1 \\
P_1(z) & = & (z - \alpha_1) \\
P_2(z) & = & (z - \alpha_1)^2 \\
P_3(z) & = & (z - \alpha_1)^2(z - \alpha_2) \\
P_4(z) & = & (z - \alpha_1)^3(z - \alpha_2) \\
P_5(z) & = & (z - \alpha_1)^3(z - \alpha_2)^2 \\
P_6(z) & = & (z - \alpha_1)^3(z - \alpha_2)^2(z - \alpha_3) \\
P_7(z) & = & (z - \alpha_1)^4(z - \alpha_2)^2(z - \alpha_3) \\
&\vdots&
\end{eqnarray*}

For convenience, let's denote $i_n=m_n+1-j_n$. For any given
$i\geq0$ and $j\geq1$ there exists a unique $n\geq1$ such that
$i_n=i$ and $j_n=j$, namely $n=\frac{(i+j)(i+j-1)}{2}+j$.
\begin{lemma}
For $n\geq 1$, we have $P_{n-1}^{(i_n)}(\alpha_{j_n})\neq0$ and
$P_{l}^{(i_n)}(\alpha_{j_n})=0$ when $l\geq n$.
\end{lemma}
\begin{proof}
From the definition of $P_n(z)$, we can write explicitly
$$P_l(z)=(z-\alpha_1)^{m_l}(z-\alpha_2)^{m_l-1}\cdots(z-\alpha_{m_l})(z-\alpha_1)\cdots(z-\alpha_{j_l})$$
It follows that $\alpha_{j_n}$ is a zero of $P_{n-1}(z)$ with
multiplicity $i_n$, which means $P_{n-1}^{(i_n)}(\alpha_{j_n})\neq
0$. On the other hand, if $l\geq n$, then $\alpha_{j_n}$ is a zero
of $P_{l}(z)$ with multiplicity at least $i_n+1$, which implies
$P_{l}^{(i_n)}(\alpha_{j_n})=0$.

\end{proof}

\begin{lemma}
If $\sum_{k=0}^{\infty}a_k P_k(z)=\sum_{k=0}^{\infty}b_k P_k(z)$ for
all $z\in \C$, then $a_k=b_k$ for each $k\geq 0$.
\end{lemma}
\begin{proof}
It suffice to prove that if $g(z):=\sum_{k=0}^{\infty}a_k P_k(z)=0$
for all $z\in \C$, then $\{a_k\}_{k\geq 0}$ is identically $0$.
Notice that $a_0=g(\alpha_1)=0$. Assuming $a_0,a_1,\ldots,a_{n-1}$
are all 0, by Lemma 2, we have
\begin{eqnarray*}
0&=&\sum_{k=0}^{\infty}a_k
P_k^{(i_{n+1})}(\alpha_{j_{n+1}})\\
&=&\sum_{k=0}^{n-1}a_k P_k^{(i_{n+1})}(\alpha_{j_{n+1}})+a_n
P_n^{(i_{n+1})}(\alpha_{j_{n+1}})+ \sum_{k=n+1}^{\infty}a_k
P_k^{(i_{n+1})}(\alpha_{j_{n+1}})\\
&=&a_n P_n^{(i_{n+1})}(\alpha_{j_{n+1}})
\end{eqnarray*}
Since $P_n^{(i_{n+1})}(\alpha_{j_{n+1}})\neq 0$, we have $a_n=0$.
Hence the proof will be completed by induction.

\end{proof}

Now we are able to prove our theorem.


\section{Proof of the theorem}
We are going to construct the desired transcendental entire function
by fixing the coefficients in the series $\sum_{k = 0}^\infty a_k
P_k(z)$ recursively, where the sequence $\{P_k\}_{k\geq 0}$ has been
defined in Section 2.

First, by Lemma 1, the condition $|a_k|\leq \frac{1}{C_k k!}$ will
ensure $\sum_{k = 0}^\infty a_k P_k(z)$ to be entire.

Now we will fix the coefficients $a_k$ recursively. For $n\geq 1$,
we denote $E_n=E_{\alpha_{j_n},i_n}$  and let the numbers
$\beta_n=\sum_{k=0}^{\infty}a_k P_k^{(i_n)}(\alpha_{j_n})$. We are
going to choose the value of $a_k$ so that $\beta_n\in
E_{\alpha_{j_n},i_n}=E_n$ for all $n\geq 1$.

By Lemma 2, we know that $P_{l}^{(i_n)}(\alpha_{j_n})=0$ when $l\geq
n$, so $\beta_n$ is actually the finite sum $\sum_{k=0}^{n-1}a_k
P_k^{(i_n)}(\alpha_{j_n})$. Notice that $\beta_1=a_0
P_0^{(0)}(\alpha_1)=a_0$ and $E_1$ is dense, we can fix a value for
$a_0$ such that $0<|a_0|\leq \frac{1}{C_0}$ and $\beta_1\in E_1$.
Now suppose that the values of $\{a_0,a_1,\cdots,a_{n-1}\}$ are well
fixed such that $0<|a_k|\leq \frac{1}{C_k k!}$ and $\beta_k\in E_k$
for $0\leq k\leq n-1$.  By Lemma 2, we know
$P_{n}^{(i_{n+1})}(\alpha_{j_{n+1}})\neq0$, so we can pick a proper
value of $a_n$ such that $0<|a_n|\leq \frac{1}{C_n n!}$ and
$\beta_n=\sum_{k=0}^{n-1}a_k P_k^{(i_{n+1})}(\alpha_{j_{n+1}})+a_n
P_{n}^{(i_{n+1})}(\alpha_{j_{n+1}})\in E_n$.

So now by induction all the $a_k$ are well chosen such that for all
$k\geq 1$ we have $0<|a_k|\leq \frac{1}{C_k k!}$ and $\beta_k\in
E_k$. Thus by Lemma 1, the function $f(z)=\sum_{k=0}^{\infty}a_k
P_k(z)$ is an entire function and for any $i\geq0,j\geq1$ we have
$f^{(i)}(\alpha_j)=\sum_{k=0}^{\infty}a_k
P_k^{(i)}(\alpha_j)=\beta_n\in E_n=E_{\alpha_j,i}$ where $n$ is the
unique integer such that $i_n=i, j_n=j$. Taking into account that
every polynomial could be expressed as a finite linear combination
of the $\{P_k\}$, and all the $\{a_k\}$ here are not 0, so by Lemma
3 we conclude that $f(z)$ is not a polynomial. Hence $f(z)$ is the
desired transcendental entire function.

\bigskip

From the construction of the proof, we can easily see that in fact
there are uncountably many functions satisfying the properties
required in Theorem 1.


\section{Applications to Exceptional Sets}

\hspace{0.5 cm}
We recall the following definition
\begin{definition}
Let $f$ be an entire function.  We define {\em the exceptional set
of $f$} to be $$S_f = \{ \alpha \in \overline{\QQ}\  |\  f(\alpha) \in \overline{\QQ}\}$$
\end{definition}

We list some of the more interesting consequences of Theorem 1 with
the choice of $A,\ E_{\alpha,s}$ noted in parentheses.
\begin{corollary}
For each countable subset $\Sigma \subseteq \C$ and each dense
subset $T \subseteq \C$ there is a transcendental entire function
$f$ such that $f^{(s)}(\Sigma) \subseteq T$ for $s \geq 0$. ($A =
\Sigma,\ E_{\alpha,s} = T$)
\end{corollary}
\begin{corollary}
Let $A \subseteq \C$ be countable and dense in $\C$, then there is a
transcendental entire function $f$ such that $f^{(s)}(A) \subseteq
A$, for $s \geq 0$. ($E_{\alpha,s} = A$)
\end{corollary}
\begin{corollary}
There exists a transcendental entire function such that
$f^{(s)}(\overline{\QQ})\subseteq \QQ(i)$, for $s\geq 0$. ($A =
\overline{\QQ},\ E_{\alpha,s} = \QQ(i)$)
\end{corollary}

The next result shows that in particular every $A \subseteq \overline{\QQ}$ is an exceptional set of a transcendental entire function.
\begin{theorem}
If $A \subseteq \overline{\QQ}$, then there is a transcendental entire function, such that $S_{f^{(s)}} = A$ for $s \geq 0$.
\end{theorem}
\begin{proof}
Suppose that $A$ and $\overline{\QQ} \backslash A$ are both
infinite, thus we can enumerate $\overline{\QQ} = \{\alpha_1,
\alpha_2, \ldots \}$ where $A = \{\alpha_1,
\alpha_3,...,\alpha_{2n+1},...\}$. Set $E_{\alpha_{2n+2},s} = \C
\backslash \overline{\QQ}$ and $E_{\alpha_{2n+1},s} =
\overline{\QQ}$ for all $n,s \geq 0$. Now by Theorem 1, there exists
a transcendental entire function $f$ with $f^{(s)}(\alpha_{2n+1})
\in \overline{\QQ}$ and $f^{(s)}(\alpha_{2n+2}) \in \C \backslash
\overline{\QQ}$, for each $n,s \geq 0$. Therefore it is plain that
$S_{f^{(s)}} = A$.

For the case that $A$ is finite, we can suppose $A =
\{\alpha_1,...,\alpha_m\}$. Take $E_{\alpha_1,s} = \cdots =
E_{\alpha_m,s} = \overline{\QQ}$ for all $s\geq 0$, and set
$E_{\alpha_k,s} = \C \backslash \overline{\QQ}$ for all $k>m, s\geq
0$. In case, $\overline{\QQ} \backslash A =
\{\alpha_1,...,\alpha_m\}$, we take $E_{\alpha_1,s} = \cdots
E_{\alpha_m,s} = \C \backslash \overline{\QQ}$ for all $s\geq 0$,
and set $E_{\alpha_k,s} = \overline{\QQ}$ for all $k>m, s\geq 0$.
Then for these two cases we proceed as in the proof above.

\end{proof}
\noindent

\vspace{0.5 cm}

\noindent
\textbf{\Large Acknowledgements}
\vspace{0.5 cm}

The authors would like to express their gratitude to Professor
Michel Waldschmidt for his instruction and guidance. We also thank
Georges Racinet for providing encouragement and assistance during
the evening workshops. Additionally, a debt is owed to all of the
students who participated; Chuangxun Cheng, Brian Dietel, Mathilde
Herblot, Holly Krieger, Jonathan Mason and Robert Wilson. Finally,
we must thank the organizers of the 2008 Arizona Winter School and
the hospitality of the University of Arizona for making this work
possible. The second author is a scholarship holder of CNPq. The
third is supported by Harrington fellowship.

%
%



\bibliographystyle{amsplain}

\begin{thebibliography}{99}

\bibitem{Mah} K. Mahler, \textit{Lectures on transcendental numbers,} Lectures notes in mathematics, Vol. 546. Springer-Verlag, Berlin-New York, 1976.



\bibitem{Su2} A. Surroca, \textit{Valeurs alg\'ebriques de fonctions transcendantes,} Int. Math. Res. Not. (2006), Art. ID 16834, 31 pages.


\bibitem{Sta1} P. St$\ddot{\mbox{a}}$ckel, \textit{Ueber arithmetische Eingenschaften analytischer Functionen,} Mathematische Annalen 46 (1895), no. 4, 513--520.

\bibitem{Sta2} P. St$\ddot{\mbox{a}}$ckel, \textit{Arithmetische Eingenschaften analytischer Functionen,} Acta Mathematica 25 (1902), 371--383.



\end{thebibliography}

\vfill

Jingjing Huang (Penn State U.)
\href{huang@math.psu.edu}{huang@math.psu.edu}

Diego Marques (U. Federal do Ceara) \href{diego@mat.ufc.br}{diego@mat.ufc.br}

Martin Mereb  (U. Texas) \href{mmereb@gmail.com}{mmereb@gmail.com}

\end{document}